\documentclass[12pt]{article}

\usepackage{amsfonts}
\usepackage{amsbsy}
\usepackage{amssymb}
\usepackage{amsmath}
\usepackage{amsthm}
\usepackage{indentfirst}
\usepackage[pdftex]{graphicx}
\usepackage{subfigure}
\usepackage{supertabular}
\usepackage{algorithm}
\usepackage{enumitem}
\usepackage{epstopdf}
\usepackage{listings}
\usepackage{color}
\usepackage{enumerate}
\usepackage{courier}
\usepackage{algorithm}
\usepackage{algpseudocode}
\usepackage{verbatim}
\usepackage[mathcal]{eucal}
\usepackage{multicol}
\usepackage{float}

\theoremstyle{plain}

\newtheorem{theorem}{Theorem}[section]
\newtheorem{lemma}{Lemma}[section]

\newtheorem{proposition}{Proposition}[section]

\theoremstyle{definition}

\newcommand{\opn}{\operatorname}
\numberwithin{equation}{section}

\begin{document}

\title{Laplace operators on holomorphic Lie algebroids}
\author{Alexandru IONESCU \footnote{Faculty of Mathematics and Computer Science, \emph{Transilvania} University of Bra\c sov,  Romania, e-mail: alexandru.codrin.ionescu@gmail.com}}
\date{}
\maketitle

\begin{abstract}
The paper introduces Laplace-type operators for functions defined on the tangent space of a Finsler Lie algebroid, using a volume form on the prolongation of the algebroid. It also presents the construction of a horizontal Laplace operator for forms defined on the prolongation of the algebroid. All of the Laplace operators considered in the paper are also locally expressed using the Chern-Finsler connection of the algebroid.
\end{abstract}

\vskip.5cm

\section*{Introduction}

The Laplacian is one of the most important and therefore intensely studied differential operator in geometry. Its main applications, in the harmonic integral and Bochner technique theories, have been analysed in the case of Riemann and K\"ahler manifolds, where the Weitzenb\"ock formulas and Hodge decomposition theorems have been obtained. In Finsler geometry, Laplacians and their applications have been mainly studied in P.L. Antonelli, B. Lackey \cite{A-L}, D. Bao, B. Lackey \cite{B-L}, O. Munteanu \cite{M.O} for the real case. In complex Finsler geometry, Laplace-type operators have been considered by C. Zhong, T. Zhong \cite{Z-Z,Z-Z1} and C. Ida \cite{Ida}.

The concept of Lie algebroid is a generalization of that of tangent bundle. Real Lie algebroids have been studied by A. Weinstein \cite{W}, P. Popescu \cite{PP1,PP2}, M. Anastasiei \cite{A}, L. Popescu \cite{P2}. Complex and holomorphic Lie algebroids have been investigated by C.-M. Marle \cite{M}, P. Popescu \cite{PP3}, P. Popescu, C. Ida \cite{I-Po}.

E. Martinez \cite{Ma1,Ma2} has introduced the notion of prolongation of a Lie algebroid, as a tool for studying the geometry of a Lie algebroid in a context which is similar to the tangent bundle of a manifold. In this paper, we use this setting in complex geometry to continue the study of holomorphic Lie algebroids from \cite{I,I-Mu,I2} for introducing Laplace-type operators for functions and for forms on a Finsler algebroid.

The first section briefly recalls notions from the geometry of Finsler Lie algebroids (\cite{I,I-Mu}) which will be used in defining the Laplace operators. The second section presents the case of Finsler algebroids (\cite{I2}), when a Chern-Finsler connection is defined from a Finsler function on the algebroid. The third section introduces two Laplacians for functions, a horizontal and a vertical one, following the ideas from the case of a complex Finsler manifold (\cite{Z-Z}). In the last section a horizontal Laplacian for forms is defined and expressed in coordinates.

Let $M$ be a complex $n$-dimensional manifold and $E$ a holomorphic vector bundle of rank $m$ over $M$. Denote by $\pi:E\rightarrow M$ the holomorphic bundle projection, by $\Gamma(E)$ the module of holomorphic sections of $\pi$ and let $T_{\mathbb{C}}M = T'M\oplus T''M$ be the complexified tangent bundle of $M$, split into the holomorphic and antiholomorphic tangent bundles.

The holomorphic vector bundle $E$ over $M$ is called anchored if there exists a holomorphic vector bundle morphism $\rho: E\rightarrow T'M$, called anchor map.

A \textit{holomorphic Lie algebroid} over $M$ is a triple $(E,[\cdot,\cdot]_E,\rho_E)$, where $E$ is a holomorphic vector bundle anchored over $M$, $[\cdot,\cdot]_E$ is a Lie bracket on $\Gamma(E)$ and $\rho_E:\Gamma(E)\rightarrow\Gamma(T'M)$ is the homomorphism of complex modules induced by the anchor map $\rho$ such that 
\begin{equation} \label{1}
[s_1,fs_2]_E = f[s_1,s_2]_E+\rho_E(s_1)(f)s_2
\end{equation}
for all $s_1,s_2\in\Gamma(E)$ and all $f\in\mathcal{H}(M)$.

As a consequence of this definition, we have that $\rho_E([s_1,s_2]_E) = [\rho_E(s_1),\\ \rho_E(s_2)]_{T'M}$ (\cite{M}), which means that $\rho_E:(\Gamma(E),[\cdot,\cdot]_E)\rightarrow(\Gamma(T'M),[\cdot,\cdot])$ is a complex Lie algebra homomorphism.

Locally, if $\{z^{k}\}_{k=\overline{1,n}}$ is a complex coordinate system on $U\subset M$ and $\{e_{\alpha }\}_{\alpha =\overline{1,m}}$ is a local frame of holomorphic sections of $E$ on $U$, then $(z^k,u^\alpha)$ are local complex coordinates on $\pi^{-1}(U)\subset E$, where $u = u^\alpha e_\alpha(z)\in E$.

The action of the holomorphic anchor map $\rho_E$ can locally be described by
\begin{equation}
\label{5}
\rho_E(e_\alpha) = \rho_\alpha^k\dfrac{\partial}{\partial z^k},
\end{equation}
while the Lie bracket $[\cdot,\cdot]_E$ is locally given by 
\begin{equation}
[e_\alpha,e_\beta]_E = \mathcal{C}^{\:\gamma}_{\alpha\beta}e_\gamma.
\label{6}
\end{equation}
The holomorphic functions $\rho^k_\alpha = \rho^k_\alpha(z)$ and $\mathcal{C}^{\:\gamma}_{\alpha\beta} = \mathcal{C}^{\:\gamma}_{\alpha\beta}(z)$ on $M$ are called \emph{the holomorphic anchor coefficients} and \emph{the holomorphic structure functions} of the Lie algebroid E, respectively.

Since $E$ is a holomorphic vector bundle, the natural complex structure acts on its sections by $J_E(e_\alpha) = ie_\alpha$ and $J_E(\bar{e}_\alpha) = -i\bar{e}_\alpha$. Hence, the complexified bundle $E_{\mathbb{C}}$ of $E$ decomposes into $E_{\mathbb{C}} = E'\oplus E''$. The local basis of sections of $E'$ is $\{e_\alpha\}_{\alpha=\overline{1,m}}$, while for $E''$, the basis is represented by $\{\bar{e}_\alpha := e_{\bar{\alpha}}\}_{\alpha=\overline{1,m}}$. Since $\rho_E:\Gamma(E)\rightarrow\Gamma(T'M)$ is a homomorphism of complex modules, it extends naturally to the complexified bundle by $\rho'(e_\alpha) = \rho_E(e_\alpha)$ and $\rho''(e_{\bar{\alpha}}) = \rho_E(e_{\bar{\alpha}})$. Thus, we can write $\rho_E = \rho'\oplus \rho''$ on the complexified bundle, and since $E$ is holomorphic, the functions $\rho(z)$ are holomorphic, hence $\rho_\alpha^{\bar{k}} = \rho_{\bar{\alpha}}^k = 0$ and $\rho_{\bar{\alpha}}^{\bar{k}} = \overline{\rho_\alpha^k}$.

As a vector bundle, the holomorphic Lie algebroid $E$ has a natural structure of complex manifold. As usual in Finsler geometry, it is of interest to consider the complexified tangent bundle $T_{\mathbb{C}}E$. Two approaches on the tangent bundle of a holomorphic Lie algebroid $E$ were described in \cite{I-Mu}. The first is the classical study of the tangent bundle of the manifold $E$, while the second is that of the prolongation of $E$. The latter idea appeared from the need of introducing geometrical objects such as nonlinear connections or sprays which could be studied in a similar manner to the tangent bundle of a complex manifold and therefore this setting seems more attractive for studying Finsler structures. 

\section{The prolongation of a holomorphic Lie algebroid}

We briefly recall here the construction of the prolongation algebroid, as defined in \cite{Ma1,Ma2} in the real case and described in detail in \cite{I,I-Mu} for the holomorphic case.

For the holomorphic Lie algebroid $E$ over a complex manifold $M$, its prolongation was introduced using the tangent mapping $\pi'_\ast: T'E\rightarrow T'M$ and the holomorphic anchor map $\rho_E: E\rightarrow T'M$. Define the subset $\mathcal{T}'E$ of $E\times T'E$ by $\mathcal{T}'E = \{(e,v)\in
E\times T'E\ |\ \rho(e) = \pi'_\ast(v)\}$ and the mapping $\pi'_{\mathcal{T}}: \mathcal{T}'E\rightarrow E$, given by $\pi'_{\mathcal{T}}(e,v) = \pi_E(v)$, where $\pi'_E: T'E\rightarrow E$ is the tangent projection. Then $(\mathcal{T}'E,\pi'_{\mathcal{T}},E)$ is a holomorphic vector bundle over $E$, of rank $2m$. For this reason, we can introduce on $\mathcal{T}'E$ some specific elements (for instance, the Chern-Finsler connection) from complex Finsler geometry. Moreover, it is easy to verify that the projection onto the second factor $\rho'_{\mathcal{T}}: \mathcal{T}'E\rightarrow T'E$, $\rho'_{\mathcal{T}}(e,v) = v$, is the
anchor of a new holomorphic Lie algebroid over the complex manifold $E$.

The holomorphic Lie algebroid $E$ has a structure of holomorphic vector bundle with respect to the complex structure $J_E$. Let $E_{\mathbb{C}}$ be the complexified bundle of $E$ and $T_{\mathbb{C}}E = T'E\oplus T''E$, its complexified tangent bundle. A similar idea to that of Martinez (\cite{Ma2}) and Popescu (\cite{P2}) in the real case leads to the definition of the complexified prolongation $\mathcal{T}_{\mathbb{C}}E$ of $E$ as follows. We extend $\mathbb{C}$-linearly the tangent mapping $\pi'_\ast: T'E\rightarrow T'M$ and the anchor $\rho_E: E\rightarrow T'M$ to obtain $\pi_{\ast,\mathbb{C}}: T_{\mathbb{C}}E\rightarrow T_{\mathbb{C}}M$ and $\rho_{E,\mathbb{C}}: E_{\mathbb{C}}\rightarrow T_{\mathbb{C}}M$, respectively. If $\pi_{E,\mathbb{C}}: T_{\mathbb{C}}E\rightarrow E_{\mathbb{C}}$ is the tangent projection extended to the complexified spaces, then we can define the subset $\mathcal{T}_{\mathbb{C}}E$ of $E_{\mathbb{C}}\times T_{\mathbb{C}}E$ by 
\begin{equation*}
\mathcal{T}_{\mathbb{C}}E = \{(e,v)\in E_{\mathbb{C}}\times T_{\mathbb{C}}E\ |\ \rho_{E,\mathbb{C}}(e) = \pi_{\ast,\mathbb{C}}(v)\}
\end{equation*}
and the mapping $\pi_{\mathcal{T},\mathbb{C}}: \mathcal{T}_{\mathbb{C}}E\rightarrow E_{\mathbb{C}}$ by $\pi_{\mathcal{T},\mathbb{C}}(e,v) = \pi_{E,\mathbb{C}}(v)$. Thus, we obtain a complex vector bundle $(\mathcal{T}_{\mathbb{C}}E,\pi_{\mathcal{T},\mathbb{C}},E_{\mathbb{C}})$ over $E_{\mathbb{C}}$. Also, the projection onto the second factor, 
\begin{equation*}
\rho_{\mathcal{T},\mathbb{C}}: \mathcal{T}_{\mathbb{C}}E \rightarrow T_{\mathbb{C}}E, \quad \rho_{\mathcal{T},\mathbb{C}}(e,v) = v,
\end{equation*}
is the anchor of a complex Lie algebroid over $E_{\mathbb{C}}$, called \emph{the complexified prolongation} of $E$.

The vertical subbundle of the complexified prolongation is defined using the projection onto the first factor $\tau_1: \mathcal{T}'E \rightarrow E$, $\tau_1(e,v) = e$, by 
\begin{equation*}
V\mathcal{T}'E = \ker\tau_1 = \{(e,v)\in\mathcal{T}'E\ |\ \tau_1(e,v) = 0\}.
\end{equation*}
Any element of $V\mathcal{T}'E$ has the form $(0,v)\in E\times \mathcal{T}'E$, with $\pi'_\ast(v) = 0$, thus vertical elements $(0,v)\in V\mathcal{T}'E$ have the property $v\in\ker\pi'_\ast$. By conjugation, we obtain $V\mathcal{T}''E$ and the complexified vertical subbundle of the prolongation $\mathcal{T}_{\mathbb{C}}E$ is $V\mathcal{T}_{\mathbb{C}}E = V\mathcal{T}'E\oplus V\mathcal{T}''E$.

The local basis of holomorphic sections in $\Gamma(\mathcal{T}'E)$ is $\{\mathcal{Z}_\alpha,\mathcal{V}_\alpha\}$, defined by 
\begin{equation*}
\mathcal{Z}_\alpha(u) = \bigg(e_\alpha(\pi(u)),\rho^k_\alpha\dfrac{\partial}{\partial z^k}\bigg|_{u}\bigg), \qquad \mathcal{V}_\alpha(u) = \bigg(0,\dfrac{\partial}{\partial u^\alpha}\bigg|_{u}\bigg),
\end{equation*}
where $\bigg\{\dfrac{\partial}{\partial z^k},\dfrac{\partial}{\partial u^\alpha}\bigg\}$ is the natural frame on $T'E$. Therefore, a local basis of sections in $\Gamma(\mathcal{T}_{\mathbb{C}}E)$ is $\{\mathcal{Z}_\alpha,\mathcal{V}_\alpha,\mathcal{Z}_{\bar{\alpha}},\mathcal{V}_{\bar{\alpha}}\}$, where $\mathcal{Z}_{\bar{\alpha}},\mathcal{V}_{\bar{\alpha}}$ are obtained by conjugation.

For a change of local charts on $E$ with the transition matrix $M$,
\begin{equation*}
\widetilde{z}^k = \widetilde{z}^k(z),\qquad \widetilde{u}^\alpha = M^\alpha_\beta(z)u^\beta,
\end{equation*}
the basis of sections on $E$, $\{e_\alpha\}$, changes by the inverse $W:=M^{-1}$,
\begin{equation} \label{sch.c.e}
\widetilde{e}_\alpha = W^\beta_\alpha e_\beta,
\end{equation}
the local coefficients of the anchor map, $\rho^k_\alpha$, change as
\begin{equation} \label{sch.c.ro}
\widetilde{\rho}^k_\alpha = W^\beta_\alpha\rho^h_\beta\dfrac{\partial \widetilde{z}^k}{\partial z^h},
\end{equation}
while the natural frame of fields $\bigg\{\dfrac{\partial}{\partial z^{k}},\dfrac{\partial}{\partial u^{\alpha}}\bigg\}$ from $T'E$ changes by the rules
\begin{align}
\dfrac{\partial}{\partial z^{h}} &= \dfrac{\partial\widetilde{z}^k}{\partial z^h}\dfrac{\partial}{\partial\widetilde{z}k}+\dfrac{\partial M_\beta^\alpha}{\partial z^h}u^\beta\dfrac{\partial}{\partial 
\widetilde{u}^\alpha}, \label{sch.c.T'E} \\
\dfrac{\partial}{\partial u^\beta} &= M_\beta^\alpha\dfrac{\partial}{\partial\widetilde{u}^\alpha},  \notag
\end{align}
see \cite{I-Mu} for more details. $E$ is a complex manifold, such that all of the above rules can also be conjugated.

The rules of change for the local basis of sections $\{\mathcal{Z}_\alpha,\mathcal{V}_\alpha,\mathcal{Z}_{\bar{\alpha}},\mathcal{V}_{\bar{\alpha}}\}$ from $\Gamma(\mathcal{T}_{\mathbb{C}}E)$ are:
\begin{align*}
\widetilde{\mathcal{Z}}_\beta &= W^\alpha_\beta\left(\mathcal{Z}_\alpha - \rho^h_\alpha\dfrac{\partial M^\gamma_\varepsilon}{\partial z^h}W^\tau_\gamma u^\varepsilon \mathcal{V}_\tau\right), \\
\widetilde{\mathcal{V}}_\beta &= W^\alpha_\beta\mathcal{V}_\alpha,
\end{align*}
together with their conjugates.

We shall further use the well-known abbreviations 
\begin{equation*}
\dfrac{\partial}{\partial z^k} := \partial_k,\ \dfrac{\partial}{\partial u^\alpha} := \dot{\partial}_\alpha,\ \dfrac{\partial}{\partial\bar{z}^k} := \partial_{\bar{k}},\ \dfrac{\partial}{\partial\bar{u}^\alpha} := \dot{\partial}_{\bar{\alpha}}.
\end{equation*}

Locally, we describe the action of the anchor map $\rho_{\mathcal{T}}$ on $\mathcal{T}E$ by 
\begin{align*}
&\rho_{\mathcal{T}}(\mathcal{Z}_\alpha) = \rho^k_\alpha\partial_k =: \partial_\alpha,\quad \rho_{\mathcal{T}}(\mathcal{V}_\alpha) = \dot{\partial}_\alpha, \\
&\rho_{\mathcal{T}}(\mathcal{Z}_{\bar{\alpha}}) = \rho^{\bar{z}}_{\bar{\alpha}}\partial_{\bar{k}}=:\partial_{\bar{\alpha}},\quad \rho_{\mathcal{T}}(\mathcal{V}_{\bar{\alpha}}) = \dot{\partial}_{\bar{\alpha}}.
\end{align*}

\section{Nonlinear connections on $\mathcal{T}'E$}

In \cite{I-Mu}, we have considered an adapted frame on $\mathcal{T}'E$ given by a complex nonlinear connection. In \cite{I2}, we have introduced a complex nonlinear connection of Chern-Finsler type on the holomorphic prolongation $\mathcal{T}'E$. Here we only recall the notions we need for defining Laplace type operators on the holomorphic Lie algebroid.

A complex nonlinear connection on $\mathcal{T}'E$ is given by a complex vector subbundle $H\mathcal{T}'E$ of $\mathcal{T}'E$ such that $\mathcal{T}'E = H\mathcal{T}'E\oplus V\mathcal{T}'E$. A local basis for the horizontal distribution $H\mathcal{T}'E$, the adapted frame of fields on $\mathcal{T}'E$, is $\{\mathcal{X}_\alpha = \mathcal{Z}_\alpha-N_\alpha^\beta\mathcal{V}_\beta,\mathcal{V}_\alpha\}$, where $N_k^\beta$ are the coefficients of a nonlinear connection on $T'E$ and $N_\alpha^\beta = \rho_\alpha^k N_k^\beta$ are functions defined on $E$, called the coefficients of the complex nonlinear connection on $\mathcal{T}'E$. We have
\begin{equation*}
\rho_{\mathcal{T}}(\mathcal{X}_\alpha) = \delta_\alpha = \rho_\alpha^k\delta_k,
\end{equation*}
where $\{\delta_k = \partial_k-N_k^\beta\partial_\beta\}$ is an adapted frame on $T'E$ (\cite{I2}).

The rules of change for the adapted frame $\{\mathcal{X}_\alpha,\mathcal{V}_\alpha\}$ are
\begin{align*}
\widetilde{\mathcal{X}}_\alpha &= W^\beta_\alpha\mathcal{X}_\beta, \\
\widetilde{\mathcal{V}}_\alpha &= W^\beta_\alpha\mathcal{V}_\beta.
\end{align*}

On the complexified prolongation bundle, a complex nonlinear connection determines the splitting of $\mathcal{T}_{\mathbb{C}}E$ as 
\begin{equation} \label{split}
\mathcal{T}_{\mathbb{C}}E = H\mathcal{T}_{\mathbb{C}}E \oplus V\mathcal{T}_{\mathbb{C}}E \oplus \overline{H\mathcal{T}_{\mathbb{C}}E} \oplus \overline{V\mathcal{T}_{\mathbb{C}}E}  
\end{equation}
such that an adapted frame $\{\mathcal{X}_\alpha,\mathcal{V}_\alpha,\mathcal{X}_{\bar{\alpha}},\mathcal{V}_{\bar{\alpha}}\}$ is obtained on $\mathcal{T}_{\mathbb{C}}E$ with respect to the complex nonlinear connection.

\begin{proposition} \label{brackets}
The Lie brackets of the adapted frame $\{\mathcal{X}_\alpha,\mathcal{V}_\alpha,\mathcal{X}_{\bar{\alpha}},\mathcal{V}_{\bar{\alpha}}\}$ are 
\begin{align*}
[\mathcal{X}_\alpha,\mathcal{X}_\beta]_{\mathcal{T}} &=
\mathcal{C}^{\:\gamma}_{\alpha\beta}\mathcal{X}_\gamma + \mathcal{R}^{\:\gamma}_{\alpha\beta}\mathcal{V}_\gamma, \\
[\mathcal{X}_\alpha,\mathcal{X}_{\bar{\beta}}]_{\mathcal{T}} &=  (\delta_{\bar{\beta}}N^\gamma_\alpha)\mathcal{V}_\gamma - (\delta_\alpha N^{\bar{\gamma}}_{\bar{\beta}})\mathcal{V}_{\bar{\gamma}}, \\
[\mathcal{X}_\alpha,\mathcal{V}_\beta]_{\mathcal{T}} &= (\dot{\partial}_\beta N^\gamma_\alpha)\mathcal{V}_\gamma, \\
[\mathcal{X}_\alpha,\mathcal{V}_{\bar{\beta}}]_{\mathcal{T}} &= (\dot{\partial}_{\bar{\beta}}N^\gamma_\alpha)\mathcal{V}_\gamma, \\
[\mathcal{V}_\alpha,\mathcal{V}_\beta]_{\mathcal{T}} &= 0, \\
[\mathcal{V}_\alpha,\mathcal{V}_{\bar{\beta}}]_{\mathcal{T}} &= 0,
\end{align*}
where 
\begin{equation*}
\mathcal{R}^{\:\gamma}_{\alpha\beta} = \mathcal{C}^{\:\varepsilon}_{\alpha\beta}N_\varepsilon^\gamma - \delta_\alpha N^\gamma_\beta + \delta_\beta N^\gamma_\alpha.
\end{equation*}
\end{proposition}

The dual of the adapted frame is $\{\mathcal{Z}^\alpha,\delta\mathcal{V}^\alpha=\mathcal{V}^\alpha+N^\alpha_\beta\mathcal{Z}^\beta\}$, where $\{\mathcal{Z}^\alpha,\mathcal{V}^\alpha\}$ is the dual frame of $\{\mathcal{Z}_\alpha,\mathcal{V}_\alpha\}$.

In \cite{I2}, following the ideas from \cite{Ai}, we have introduced the Chern-Finsler nonlinear connection of the prolongation $\mathcal{T}'E$. If $F:E\rightarrow\mathbb{R}_+$ is a Finsler function on $E$ (\cite{I2}), i.e. it is homogeneous, and the complex Finsler metric tensor
\begin{equation*}
h_{\alpha \bar{\beta}}=\dot{\partial}_\alpha \dot{\partial}_{\bar{\beta}}F,
\end{equation*}
is strictly pseudoconvex, then
\begin{equation*}
N_\alpha^\beta = h^{\bar{\sigma}\beta }\partial_\alpha\dot{\partial}_{\bar{\sigma}}F
\end{equation*}
are the coefficients of the Chern-Finsler nonlinear connection of the prolongation $\mathcal{T}'E$. Also, a Chern-Finsler linear connection of type $(1,0)$ on $\mathcal{T}'E$ is given by
\begin{equation*}
L^{\:\gamma}_{\alpha\beta} = h^{\bar{\sigma}\gamma}\delta_\beta(h_{\alpha\bar{\sigma}}),\quad C^{\:\gamma}_{\alpha\beta} = h^{\bar{\sigma}\gamma}\dot{\partial}_\beta(h_{\alpha\bar{\sigma}}).
\end{equation*}
Its connection form is
\begin{equation*}
\omega^\gamma_\alpha = L^{\:\gamma}_{\alpha\beta}\mathcal{Z}^\beta + C^{\:\gamma}_{\alpha\beta}\delta\mathcal{V}^\beta.
\end{equation*}
Also, we note that
\begin{equation} \label{C}
C^{\:\gamma}_{\alpha\beta} = C^{\:\gamma}_{\beta\alpha}
\end{equation}
and
\begin{equation} \label{L,N}
L^{\:\gamma}_{\alpha\beta} = \dot{\partial}_\alpha N^\gamma_\beta.
\end{equation}
The prolongation algebroid $\mathcal{T}'E$ is called K\"ahler Finsler algebroid if
\begin{equation} \label{cond.K}
L^{\:\sigma}_{\alpha\gamma} = L^{\:\sigma}_{\gamma\alpha},
\end{equation}
see also \cite{I2} for more details. 

If we denote by $h = \det(h_{\alpha\bar{\beta}})$, then using similar reasons as in the case of a complex Finsler manifold (\cite{Z-Z}) we get
\begin{equation} \label{L,C det}
L^{\:\beta}_{\beta\alpha} = \delta_\alpha(\opn{ln}h),\quad C^{\:\beta}_{\beta\alpha} = \dot{\partial}_\alpha(\opn{ln}h).
\end{equation}

A metric structure on the complexified prolongation $\mathcal{T}_{\mathbb{C}}E$ is given by
\begin{equation} \label{metric}
\mathcal{G} = h_{\alpha\bar{\beta}}\mathcal{Z}^\alpha\otimes\bar{\mathcal{Z}}^\beta + h_{\alpha\bar{\beta}}\delta\mathcal{V}^\alpha\otimes\delta\bar{\mathcal{V}}^\beta.
\end{equation}

Next, we shall express the covariant derivatives of tensor fields on $\mathcal{T}'E$ with respect to the Chern-Finsler connection, following the ideas from the case of a complex Finsler manifold (\cite{Z-Z}). A complex horizontal covariant tensor field is given by
\begin{equation*}
T = \dfrac{1}{p!q!}T_{\alpha_1\dots\alpha_p\bar{\beta}_1\dots\bar{\beta}_q}(z,u)\mathcal{Z}^{\alpha_1}\wedge\dots\wedge\mathcal{Z}^{\alpha_p}\wedge\mathcal{Z}^{\bar{\beta}_1}\wedge\dots\wedge\mathcal{Z}^{\bar{\beta}_q},
\end{equation*}
where for the changes \eqref{sch.c.e}, the local components $T_{\alpha_1\dots\alpha_p\bar{\beta}_1\dots\bar{\beta}_q}(z,u)$ change by the rules
\begin{equation*}
\widetilde{T}_{\alpha_1\dots\alpha_p\bar{\beta}_1\dots\bar{\beta}_q}(\widetilde{z},\widetilde{u}) = T_{\gamma_1\dots\gamma_p\bar{\varepsilon}_1\dots\bar{\varepsilon}_q}M^{\gamma_1}_{\alpha_1}\dots M^{\gamma_p}_{\alpha_p}M^{\bar{\varepsilon}_1}_{\bar{\beta}_1}\dots M^{\bar{\varepsilon}_q}_{\bar{\beta}_q}.
\end{equation*}

Similarly, we can define a horizontal contravariant tensor field, whose local components,  $T^{\alpha_1\dots\alpha_p\bar{\beta}_1\dots\bar{\beta}_q}(z,u)$, change as
\begin{equation*}
\widetilde{T}^{\alpha_1\dots\alpha_p\bar{\beta}_1\dots\bar{\beta}_q}(\widetilde{z},\widetilde{u}) = T^{\gamma_1\dots\gamma_p\bar{\varepsilon}_1\dots\bar{\varepsilon}_q}W_{\gamma_1}^{\alpha_1}\dots W_{\gamma_p}^{\alpha_p}W_{\bar{\varepsilon}_1}^{\bar{\beta}_1}\dots W_{\bar{\varepsilon}_q}^{\bar{\beta}_q}.
\end{equation*}

The differential of a function $f$ on the complexified prolongation $\mathcal{T}_{\mathbb{C}}E$ is locally expressible as
\begin{equation*}
df = (\delta_\alpha f)\mathcal{Z}^\alpha + (\dot{\partial}_\alpha f)\delta\mathcal{V}^\alpha + (\delta_{\bar{\alpha}}f)\mathcal{Z}^{\bar{\alpha}} + (\dot{\partial}_{\bar{\alpha}}f)\delta\mathcal{V}^{\bar{\alpha}}.
\end{equation*}
With respect to the \eqref{split} decomposition of the prolongation, the differential can be written as
\begin{equation*}
df = \partial^hf + \partial^vf + \bar{\partial}^hf + \bar{\partial}^vf,
\end{equation*}
where
\begin{align*}
\partial^hf &= (\delta_\alpha f)\mathcal{Z}^\alpha = \bigg(\rho^k_\alpha\dfrac{\partial f}{\partial z^k} - N^\beta_\alpha\dfrac{\partial f}{\partial u^\beta}\bigg)\mathcal{Z}^\alpha,\quad
\partial^vf &= (\dot{\partial}_\alpha f)\delta\mathcal{V}^\alpha = \dfrac{\partial f}{\partial u^\alpha}\delta\mathcal{V}^\alpha,\\
\bar{\partial}^hf &= (\delta_{\bar{\alpha}}f)\mathcal{Z}^{\bar{\alpha}} = \bigg(\rho^{\bar{k}}_{\bar{\alpha}}\dfrac{\partial f}{\partial\bar{z}^k} - N^{\bar{\beta}}_{\bar{\alpha}}\dfrac{\partial f}{\partial\bar{u}^\beta}\bigg)\mathcal{Z}^{\bar{\alpha}},\quad
\bar{\partial}^vf &= (\dot{\partial}_{\bar{\alpha}}f)\delta\mathcal{V}^{\bar{\alpha}} = \dfrac{\partial f}{\partial\bar{u}^\alpha}\delta\mathcal{V}^{\bar{\alpha}}.
\end{align*}
In particular,
\begin{equation*}
d\mathcal{Z}^\alpha = -\dfrac{1}{2}C^{\:\alpha}_{\beta\gamma}\mathcal{Z}^\beta\wedge\mathcal{Z}^\gamma  - \dfrac{1}{2}C^{\:\alpha}_{\bar{\beta}\gamma}\mathcal{Z}^{\bar{\beta}}\wedge\mathcal{Z}^\gamma,\qquad d\mathcal{V}^\alpha = 0.
\end{equation*}

We shall now restrict our considerations on the horizontal bundle $H\mathcal{T}E$ of the prolongation and describe the horizontal derivatives of tensors with respect to the Chern-Finsler connection of the prolongation.

First, we define in a classical manner the horizontal covariant derivative of a horizontal covariant tensor field $T_{\alpha_1\dots\alpha_p\bar{\beta}_1\dots\bar{\beta}_q}(z,u)$ as
\begin{align*}
\nabla_{\mathcal{X}_\gamma}T_{\alpha_1\dots\alpha_p\bar{\beta}_1\dots\bar{\beta}_q} &= \mathcal{X}_\gamma(T_{\alpha_1\dots\alpha_p\bar{\beta}_1\dots\bar{\beta}_q}) - \sum_{i=1}^p T_{\alpha_1\dots\alpha_{i-1}\varepsilon\alpha_{i+1}\dots\alpha_p\bar{\beta}_1\dots\bar{\beta}_q} L^{\:\varepsilon}_{\alpha_i\gamma},\\
\nabla_{\mathcal{X}_{\bar{\gamma}}}T_{\alpha_1\dots\alpha_p\bar{\beta}_1\dots\bar{\beta}_q} &= \mathcal{X}_{\bar{\gamma}}(T_{\alpha_1\dots\alpha_p\bar{\beta}_1\dots\bar{\beta}_q}) - \sum_{j=1}^q T_{\alpha_1\dots\alpha_p\bar{\beta}_1\dots\bar{\beta}_{j-1}\bar{\varepsilon}\bar{\beta}_{j+1}\dots\bar{\beta}_q} L^{\:\bar{\varepsilon}}_{\bar{\beta}_j\bar{\gamma}}.
\end{align*}

Further, the horizontal covariant derivative of a contravariant tensor\\ $T^{\alpha_1\dots\alpha_p\bar{\beta}_1\dots\bar{\beta}_q}(z,u)$ is defined by
\begin{align*}
\nabla_{\mathcal{X}_\gamma}T^{\alpha_1\dots\alpha_p\bar{\beta}_1\dots\bar{\beta}_q} &= \mathcal{X}_\gamma(T^{\alpha_1\dots\alpha_p\bar{\beta}_1\dots\bar{\beta}_q}) + \sum_{i=1}^p T^{\alpha_1\dots\alpha_{i-1}\varepsilon\alpha_{i+1}\dots\alpha_p\bar{\beta}_1\dots\bar{\beta}_q} L^{\:\alpha_i}_{\varepsilon\gamma},\\
\nabla_{\mathcal{X}_{\bar{\gamma}}}T^{\alpha_1\dots\alpha_p\bar{\beta}_1\dots\bar{\beta}_q} &= \mathcal{X}_{\bar{\gamma}}(T^{\alpha_1\dots\alpha_p\bar{\beta}_1\dots\bar{\beta}_q}) + \sum_{j=1}^q T^{\alpha_1\dots\alpha_p\bar{\beta}_1\dots\bar{\beta}_{j-1}\bar{\varepsilon}\bar{\beta}_{j+1}\dots\bar{\beta}_q} L^{\:\bar{\beta}_j}_{\bar{\varepsilon}\bar{\gamma}}.
\end{align*}
The vertical covariant derivatives can be defined in a similar manner.

\section{Vertical and horizontal Laplace type\\ operators for functions on $E$}

In this section, we shall define vertical and horizontal Laplace type operators for functions on the prolongation $\mathcal{T}E$ (we drop the index $\mathbb{C}$), following the ideas from the case of complex Finsler bundles (\cite{Z-Z,Ida}). For this purpose, we need to define the divergence of a vector field on $\mathcal{T}E$ and the gradient of a function on $\mathcal{T}E$.

First, we consider the Hermitian form associated to the metric structure $\mathcal{G}$ from \eqref{metric},
\begin{equation} \label{f.H.prel}
\Phi = i h_{\alpha\bar{\beta}}\big(\mathcal{Z}^\alpha\wedge\mathcal{Z}^{\bar{\beta}} + \delta\mathcal{V}^\alpha\wedge\delta\mathcal{V}^{\bar{\beta}}\big) = \Phi^h + \Phi^v.
\end{equation}

Denote by
\begin{align*}
(\Phi^h)^m &= i^m(-1)^{\frac{m(m-1)}{2}}m!\:h\:\mathcal{Z}^1\wedge\dots\wedge\mathcal{Z}^m\wedge\mathcal{Z}^{\bar{1}}\wedge\dots\wedge\mathcal{Z}^{\bar{m}},\\
(\Phi^v)^m &= i^m(-1)^{\frac{m(m-1)}{2}}m!\:h\:\delta\mathcal{V}^1\wedge\dots\wedge\delta\mathcal{V}^m\wedge\delta\mathcal{V}^{\bar{1}}\wedge\dots\wedge\delta\mathcal{V}^{\bar{m}},
\end{align*}
such that we can associate with $\mathcal{G}$ a volume form on $\mathcal{T}E$ by
\begin{equation}
\label{f.vol}
d\mathcal{V} = \dfrac{1}{(2m)!}\Phi^{2m} = i^{2m^2}h^2\:\mathcal{Z}\wedge\bar{\mathcal{Z}}\wedge\delta\mathcal{V}\wedge\overline{\delta\mathcal{V}},
\end{equation}
where
\begin{equation*}
\mathcal{Z} = \mathcal{Z}^1\wedge\dots\wedge\mathcal{Z}^m,\quad \delta\mathcal{V} = \delta\mathcal{V}^1\wedge\dots\wedge\delta\mathcal{V}^m
\end{equation*}
and their conjugates.

Let $Z = Z^\alpha\mathcal{X}_\alpha + V^\alpha\mathcal{V}_\alpha + Z^{\bar{\alpha}}\mathcal{X}_{\bar{\alpha}} + V^{\bar{\alpha}}\mathcal{V}_{\bar{\alpha}} \in \Gamma(\mathcal{T}E)$. The divergence of $Z$ is defined by the classical equation
\begin{equation*}
\mathcal{L}_Zd\mathcal{V} = (\opn{div}Z)d\mathcal{V},
\end{equation*}
where $\mathcal{L}_Z$ is the Lie derivative.

The expression of $Z$ according to the splitting \eqref{split} gives the following decomposition of the divergence of $Z$:
\begin{equation*}
\opn{div}Z = \opn{div}^hZ + \opn{div}^vZ + \opn{div}^{\bar{h}}Z + \opn{div}^{\bar{v}}Z,
\end{equation*}
where $\opn{div}^hZ=\opn{div}Z^h,\ \opn{div}^vZ=\opn{div}Z^v,\ \opn{div}^{\bar{h}}Z=\opn{div}Z^{\bar{h}},\ \opn{div}^{\bar{v}}Z=\opn{div}Z^{\bar{v}}$. In particular, on $\mathcal{T}'E$, we have

\begin{proposition}
The components of the divergence of $Z = Z^\alpha\mathcal{X}_\alpha + V^\alpha\mathcal{V}_\alpha \in \Gamma(\mathcal{T}'E)$ are
\begin{align}
\label{div}
\opn{div}^hZ &= \nabla_{\mathcal{X}_\alpha}Z^\alpha - Z^\alpha L_\alpha - Z^\alpha\mathcal{C}_\alpha,\\
\nonumber
\opn{div}^vZ &= \nabla_{\mathcal{V}_\alpha}V^\alpha + V^\alpha C_\alpha,\\
\end{align}
where we have denoted $L_\alpha=L^{\:\beta}_{\alpha\beta}-L^{\:\beta}_{\beta\alpha}$, $C_\alpha=C^{\:\beta}_{\alpha\beta}=C^{\:\beta}_{\beta\alpha}$ and $\mathcal{C}_\alpha=\mathcal{C}^{\:\beta}_{\alpha\beta}$.
\end{proposition}
\begin{proof}
Using Proposition \ref{brackets} and \eqref{L,N}, we obtain:
\begin{align*}
[Z,\mathcal{X}_\beta] &= \big(Z^\alpha\mathcal{C}^{\:\gamma}_{\alpha\beta}-\delta_\beta(Z^\gamma)\big)\mathcal{X}_\gamma + \big(Z^\alpha\mathcal{R}^{\:\gamma}_{\alpha\beta}-\delta_\beta(V^\gamma) - V^\alpha L^{\:\gamma}_{\alpha\beta}\big)\mathcal{V}_\gamma,\\
[Z,\mathcal{X}_{\bar{\beta}}] &= -\delta_{\bar{\beta}}(Z^\gamma)\mathcal{X}_\gamma + \big(Z^\alpha\delta_{\bar{\beta}}(N^\gamma_\alpha) - \delta_{\bar{\beta}}(V^\gamma)\big)\mathcal{V}_\gamma\\
&\quad - \big(Z^\alpha\delta_\alpha(N^{\bar{\gamma}}_{\bar{\beta}}) + V^\alpha\dot{\partial}_\alpha (N^{\bar{\gamma}}_{\bar{\beta}})\big)\mathcal{V}_{\bar{\gamma}},\\
[Z,\mathcal{V}_\beta] &= -\dot{\partial}_\beta(Z^\gamma)\mathcal{X}_\gamma + \big(Z^\alpha L^{\:\gamma}_{\beta\alpha} - \dot{\partial}_\beta(V^\gamma)\big)\mathcal{V}_\gamma,\\
[Z,\mathcal{V}_{\bar{\beta}}] &= -\dot{\partial}_{\bar{\beta}}(Z^\gamma)\mathcal{X}_\gamma + \big(Z^\alpha\dot{\partial}_{\bar{\beta}}(N^\gamma_\alpha) - \dot{\partial}_{\bar{\beta}}(V^\gamma)\big)\mathcal{V}_\gamma.
\end{align*}

Then, from the definitions of the covariant derivatives, we get
\begin{align*}
\nabla_{\mathcal{X}_\alpha}Z^\alpha &= \delta_\alpha(Z^\alpha)+Z^\alpha L^{\:\beta}_{\alpha\beta},\\
\nabla_{\mathcal{V}_\alpha}V^\alpha &= \dot{\partial}_\alpha(V^\alpha)+V^\alpha C^{\:\beta}_{\alpha\beta}.
\end{align*}

From the definition of the divergence, we have
\begin{align*}
(\opn{div}Z)h^2 &= (\opn{div}Z)d\mathcal{V}(\mathcal{X}_1,\dots,\mathcal{X}_m,\mathcal{X}_{\bar{1}},\dots,\mathcal{X}_{\bar{m}},\mathcal{V}_1,\dots,\mathcal{V}_m,\mathcal{V}_{\bar{1}},\dots,\mathcal{V}_{\bar{m}})\\
&= (\mathcal{L}_Zd\mathcal{V})(\mathcal{X}_1,\dots,\mathcal{X}_m,\mathcal{X}_{\bar{1}},\dots,\mathcal{X}_{\bar{m}},\mathcal{V}_1,\dots,\mathcal{V}_m,\mathcal{V}_{\bar{1}},\dots,\mathcal{V}_{\bar{m}})\\
&= Z(d\mathcal{V}((\mathcal{X}_1,\dots,\mathcal{X}_m,\mathcal{X}_{\bar{1}},\dots,\mathcal{X}_{\bar{m}},\mathcal{V}_1,\dots,\mathcal{V}_m,\mathcal{V}_{\bar{1}},\dots,\mathcal{V}_{\bar{m}})))\\
&\quad -\sum_{\beta=1}^m d\mathcal{V}(\mathcal{X}_1,\dots,\mathcal{X}_{\beta-1},[Z,\mathcal{X}_\beta],\mathcal{X}_{\beta+1},\dots,\mathcal{X}_m,\dots,\dots,\dots)\\
&\quad -\sum_{\beta=1}^m d\mathcal{V}(\dots,\mathcal{X}_{\bar{1}},\dots,\mathcal{X}_{\overline{\beta-1}},[Z,\mathcal{X}_{\bar{\beta}}],\mathcal{X}_{\overline{\beta+1}},\dots,\mathcal{X}_{\bar{m}},\dots,\dots)\\
&\quad -\sum_{\beta=1}^m d\mathcal{V}(\dots,\dots,\mathcal{V}_1,\dots,\mathcal{V}_{\beta-1},[Z,\mathcal{V}_\beta],\mathcal{V}_{\beta+1},\dots,\mathcal{V}_m,\dots)\\
&\quad -\sum_{\beta=1}^m d\mathcal{V}(\dots,\dots,\dots,\mathcal{V}_{\bar{1}},\dots,\mathcal{V}_{\overline{\beta-1}},[Z,\mathcal{V}_{\bar{\beta}}],\mathcal{V}_{\overline{\beta+1}},\dots,\mathcal{V}_{\bar{m}})\\
&= Z(h^2) - \big(Z^\alpha\mathcal{C}^{\:\beta}_{\alpha\beta} - \delta_\beta(Z^\beta) + Z^\alpha L^{\:\beta}_{\beta\alpha} + \dot{\partial}_\beta(V^\beta)\big)h^2.
\end{align*}
Hence,
\begin{align*}
\opn{div}^hZ &= h^{-2}Z^\alpha\mathcal{X}_\alpha(h^2) - Z^\alpha\mathcal{C}^{\:\beta}_{\alpha\beta} + \delta_\beta(Z^\beta) - Z^\alpha L^{\:\beta}_{\beta\alpha}\\
&= 2Z^\alpha h^{-1}\delta_\alpha(h) - Z^\alpha\mathcal{C}^{\:\beta}_{\alpha\beta} + \delta_\beta(Z^\beta) - Z^\alpha L^{\:\beta}_{\beta\alpha}\\
&= 2Z^\alpha\delta_\alpha(ln\:h) - Z^\alpha\mathcal{C}^{\:\beta}_{\alpha\beta} + \delta_\beta(Z^\beta) - Z^\alpha L^{\:\beta}_{\beta\alpha}\\
&= Z^\alpha L^{\:\beta}_{\beta\alpha} - Z^\alpha\mathcal{C}^{\:\beta}_{\alpha\beta} + \delta_\alpha(Z^\alpha)\\
&= \nabla_{\mathcal{X}_\alpha}Z^\alpha - Z^\alpha L_\alpha - Z^\alpha\mathcal{C}_\alpha,\\
\opn{div}^vZ &= h^{-2}V^\alpha\mathcal{V}_\alpha(h^2) + \dot{\partial}_\beta(V^\beta)\\
&= 2V^\alpha\dot{\partial}_\alpha(ln\:h) + \dot{\partial}_\beta(V^\beta)\\
&= 2V^\alpha C^{\:\beta}_{\beta\alpha} + \dot{\partial}_\alpha(V^\alpha)\\
&= \nabla_{\mathcal{V}_\alpha}V^\alpha + V^\alpha C_\alpha.
\end{align*}
\end{proof}

Note that, for a K\"ahler Finsler algebroid, the condition \eqref{cond.K} yields $L_\alpha=0$, thus
\begin{equation*}
\opn{div}^hZ = \nabla_{\mathcal{X}_\alpha}X^\alpha - X^\alpha\mathcal{C}_\alpha.
\end{equation*}

The following step is defining the gradient of a function, which can be introduced in a classical manner by
\begin{equation*}
\mathcal{G}(Z,\opn{grad}f) = Zf,\quad \forall Z\in\Gamma(\mathcal{T}'E),
\end{equation*}
and decomposing it in the adapted frame of $\mathcal{T}'E$ as
\begin{equation*}
\opn{grad}f = \opn{grad}^hf + \opn{grad}^vf,
\end{equation*}
where
\begin{equation} \label{grad}
\opn{grad}^hf = h^{\bar{\gamma}\alpha}(\delta_{\bar{\gamma}}f)\mathcal{X}_\alpha,\quad \opn{grad}^vf = h^{\bar{\varepsilon}\beta}(\dot{\partial}_{\bar{\varepsilon}}f)\mathcal{V}_\beta.
\end{equation}

We will define two Laplace operators for functions, a horizontal and a vertical one. The horizontal Laplace operator for functions on the prolongation algebroid is
\begin{equation} \label{Lhf}
\Delta^hf = (\opn{div}^h\circ\opn{grad}^h)f,
\end{equation}
and the vertical one is
\begin{equation} \label{Lvf}
\Delta^vf = (\opn{div}^v\circ\opn{grad}^v)f.
\end{equation}

The expressions for the two Laplace operators are given in the following
\begin{proposition}
For a function $f\in C^\infty(E)$, we have
\begin{equation} \label{Lhfc}
\Delta^hf = \dfrac{1}{h}\delta_\alpha\big[hh^{\bar{\gamma}\alpha}(\delta_{\bar{\gamma}}f)\big] - \big[h^{\bar{\gamma}\alpha}(\delta_{\bar{\gamma}}f)\big]\mathcal{C}_\alpha
\end{equation}
and
\begin{equation} \label{Lvfc}
\Delta^vf = \dfrac{1}{h}\dot{\partial}_\alpha\big[hh^{\bar{\gamma}\alpha}(\dot{\partial}_{\bar{\gamma}}f)\big] + \big[h^{\bar{\gamma}\alpha}(\dot{\partial}_{\bar{\gamma}}f)\big]C_\alpha.
\end{equation}
\end{proposition}
\begin{proof}
Using \eqref{div}, \eqref{grad}, \eqref{Lhf}, \eqref{Lvf} and \eqref{L,C det} we obtain:
\begin{align*}
\Delta^hf &= \nabla_{\mathcal{X}_\alpha}\big[h^{\bar{\gamma}\alpha}(\delta_{\bar{\gamma}}f)\big] - h^{\bar{\gamma}\alpha}(\delta_{\bar{\gamma}}f)L_\alpha - h^{\bar{\gamma}\alpha}(\delta_{\bar{\gamma}}f)\mathcal{C}_\alpha\\
&= \delta_\alpha\big[h^{\bar{\gamma}\alpha}(\delta_{\bar{\gamma}}f)\big] + h^{\bar{\gamma}\alpha}(\delta_{\bar{\gamma}}f)L^{\:\beta}_{\beta\alpha} - h^{\bar{\gamma}\alpha}(\delta_{\bar{\gamma}}f)\mathcal{C}_\alpha\\
&= \dfrac{1}{h}\delta_\alpha\big[hh^{\bar{\gamma}\alpha}(\delta_{\bar{\gamma}}f)\big] - h^{\bar{\gamma}\alpha}(\delta_{\bar{\gamma}}f)\mathcal{C}_\alpha.
\end{align*}
Also,
\begin{align*}
\Delta^vf &= \nabla_{\mathcal{V}_\alpha}\big[h^{\bar{\gamma}\alpha}(\dot{\partial}_{\bar{\gamma}}f)\big] + h^{\bar{\gamma}\alpha}(\dot{\partial}_{\bar{\gamma}}f)C_\alpha\\
&= \dot{\partial}_\alpha\big[h^{\bar{\gamma}\alpha}(\dot{\partial}_{\bar{\gamma}}f)\big] + 2h^{\bar{\gamma}\alpha}(\dot{\partial}_{\bar{\gamma}}f)C_\alpha\\
&= \dfrac{1}{h}\dot{\partial}_\alpha\big[hh^{\bar{\gamma}\alpha}(\dot{\partial}_{\bar{\gamma}}f)\big] + h^{\bar{\gamma}\alpha}(\dot{\partial}_{\bar{\gamma}}f)C_\alpha.
\end{align*}
\end{proof}

We note that the two Laplacian operators can also be expressed in terms of the covariant derivatives with respect to the Chern-Finsler connection as follows:
\begin{align*}
\Delta^hf &= h^{\bar{\gamma}\alpha}\left[\nabla_{\mathcal{X}_\alpha}\nabla_{\mathcal{X}_{\bar{\gamma}}}f - \mathcal{C}_\alpha\left(\nabla_{\mathcal{X}_{\bar{\gamma}}}f\right)\right]\\
\Delta^vf &= h^{\bar{\gamma}\alpha}\left[\nabla_{\mathcal{V}_\alpha}\nabla_{\mathcal{V}_{\bar{\gamma}}}f + C_\alpha\left(\nabla_{\mathcal{V}_{\bar{\gamma}}}f\right)\right]
\end{align*}

In the end of this section, following \cite{Z-Z}, we will prove a result which will be used in the last section to obtain the expressions of a horizontal Laplace operator for forms.

\begin{lemma}
If $Z=Z^\alpha\mathcal{X}_\alpha\in\Gamma(H\mathcal{T}'E)$, then
\begin{equation}
\label{div.dV}
(\opn{div}Z + \mathcal{C}^h)d\mathcal{V} = d[i_Zd\mathcal{V}],\quad (\opn{div}\bar{Z} + \bar{\mathcal{C}}^h)d\mathcal{V} = d[i_{\bar{Z}}d\mathcal{V}],
\end{equation}
where $\mathcal{C}^h = Z^\alpha\mathcal{C}_\alpha = Z^\alpha\mathcal{C}^{\:\beta}_{\alpha\beta}$.
\end{lemma}
\begin{proof}
First,
\begin{equation*}
i_Zd\mathcal{V} = \sum_{\alpha=1}^m (-1)^{\alpha-1}Z^\alpha h^2\mathcal{Z}^{\hat{\alpha}}\wedge\overline{\mathcal{Z}}\wedge\delta\mathcal{V}\wedge\overline{\delta\mathcal{V}},
\end{equation*}
where we have denoted $\mathcal{Z}^{\hat{\alpha}}=\mathcal{Z}^1\wedge\dots\wedge\widehat{\mathcal{Z}}^\alpha\wedge\dots\wedge\mathcal{Z}^m$.
Also,
\begin{align*}
d(\delta\mathcal{V}^\beta) &= d(\mathcal{V}^\beta+N^\beta_\gamma\mathcal{Z}^\gamma)\\
&= \delta_\alpha(N^\beta_\gamma)\mathcal{Z}^\alpha\wedge\mathcal{Z}^\gamma + \delta_{\bar{\alpha}}(N^\beta_\gamma)\mathcal{Z}^{\bar{\alpha}}\wedge\mathcal{Z}^\gamma + \dot{\partial}_\alpha(N^\beta_\gamma)\delta\mathcal{V}^\alpha\wedge\mathcal{Z}^\gamma\\
&\quad + \dot{\partial}_{\bar{\alpha}}(N^\beta_\gamma)\delta\mathcal{Z}^{\bar{\alpha}}\wedge\mathcal{Z}^\gamma - \dfrac{1}{2}N^\beta_\gamma\mathcal{C}^{\:\gamma}_{\alpha\varepsilon}\mathcal{Z}^\alpha\wedge\mathcal{Z}^\varepsilon - \dfrac{1}{2}N^\beta_\gamma\mathcal{C}^{\:\gamma}_{\bar{\alpha}\varepsilon}\mathcal{Z}^{\bar{\alpha}}\wedge\mathcal{Z}^\varepsilon\\
&= \bigg[\delta_\alpha(N^\beta_\gamma) - \dfrac{1}{2}N^\beta_\varepsilon\mathcal{C}^{\:\varepsilon}_{\alpha\gamma}\bigg]\mathcal{Z}^\alpha\wedge\mathcal{Z}^\gamma + \bigg[\delta_{\bar{\alpha}}(N^\beta_\gamma) - \dfrac{1}{2}N^\beta_\varepsilon\mathcal{C}^{\:\varepsilon}_{\bar{\alpha}\gamma}\bigg]\mathcal{Z}^{\bar{\alpha}}\wedge\mathcal{Z}^\gamma\\
&\quad + \dot{\partial}_\alpha(N^\beta_\gamma)\delta\mathcal{V}^\alpha\wedge\mathcal{Z}^\gamma + \dot{\partial}_{\bar{\alpha}}(N^\beta_\gamma)\delta\mathcal{Z}^{\bar{\alpha}}\wedge\mathcal{Z}^\gamma.
\end{align*}
Thus,
\begin{align*}
d[i_Zd\mathcal{V}] &= \sum_\alpha\sum_\beta(-1)^{\alpha-1}\delta_\beta(Z^\alpha h^2)\mathcal{Z}^\beta\wedge\mathcal{Z}^{\hat{\alpha}}\wedge\overline{\mathcal{Z}}\wedge\delta\mathcal{V}\wedge\overline{\delta\mathcal{V}}\\
&\quad +\sum_\alpha\sum_\beta(-1)^{\alpha-1+2m-1+\beta-1}Z^\alpha h^2\mathcal{Z}^{\hat{\alpha}}\wedge\overline{\mathcal{Z}}\\
&\quad \wedge\delta\mathcal{V}^1\wedge\dots\wedge d(\delta\mathcal{V}^\beta)\wedge\dots\wedge\delta\mathcal{V}^m\wedge\overline{\delta\mathcal{V}}\\
&= [\delta_\alpha(Z^\alpha h^2)-Z^\alpha L^{\:\beta}_{\beta\alpha}h^2]\mathcal{Z}\wedge\overline{\mathcal{Z}}\wedge\delta\mathcal{V}\wedge\overline{\delta\mathcal{V}}.
\end{align*}
But
\begin{align*}
\delta_\alpha(Z^\alpha h^2)-Z^\alpha L^{\:\beta}_{\beta\alpha}h^2 &= [\delta_\alpha(Z^\alpha)-2Z^\alpha\delta_\alpha(\opn{ln}h)-Z^\alpha L^{\:\beta}_{\beta\alpha}]h^2\\
&= [\delta_\alpha(Z^\alpha)+Z^\alpha L^{\:\beta}_{\beta\alpha}]h^2\\
&= [\nabla_{\mathcal{X}_\alpha}Z^\alpha - Z^\alpha L_\alpha]h^2\\
&= [\opn{div}^hZ+Z^\alpha L_\alpha]h^2,
\end{align*}
such that the first identity is proved. The second identity can be obtained by conjugation.
\end{proof}

Using this and \eqref{div}, we get
\begin{proposition}
If $Z=Z^\alpha\mathcal{X}_\alpha$ is a horizontal field with compact support on the prolongation of a Finsler algebroid, then
\begin{equation}
\label{int1}
\int_E (\nabla_{\mathcal{X}_\alpha}Z^\alpha - Z^\alpha L_\alpha)d\mathcal{V} = 0,\quad \int_E (\nabla_{\mathcal{X}_{\bar{\alpha}}}\overline{Z^\alpha} - \overline{Z^\alpha L_\alpha})d\mathcal{V} = 0.
\end{equation}
\end{proposition}

In the case of K\"ahler Finsler algebroids, \eqref{int1} becomes
\begin{equation}
\label{int2}
\int_E \nabla_{\mathcal{X}_\alpha}Z^\alpha d\mathcal{V} = 0,\quad \int_E \nabla_{\mathcal{X}_{\bar{\alpha}}}\overline{Z^\alpha} d\mathcal{V} = 0.
\end{equation}

\section{A horizontal Laplace operator for forms on the prolongation algebroid $\mathcal{T}E$}

In this section we define a horizontal Laplace-type operator for forms with compact support defined on the prolongation of a Finsler algebroid.

Consider two horizontal forms with compact support on $\mathcal{T}E$, $\Psi$ and $\Phi$, locally defined by
\begin{align*}
\Psi &= \dfrac{1}{p!q!}\psi_{A_p\bar{B}_q}\mathcal{Z}^{A_p}\wedge\mathcal{Z}^{\bar{B}_q},\\
\Phi &= \dfrac{1}{p!q!}\phi_{A_p\bar{B}_q}\mathcal{Z}^{A_p}\wedge\mathcal{Z}^{\bar{B}_q},
\end{align*}
where we have denoted the multi-indices $A_p = (\alpha_1\dots\alpha_p)$, $\bar{B}_q = (\bar{\beta}_1\dots\bar{\beta}_q)$ and $\mathcal{Z}^{A_p} = \mathcal{Z}^{\alpha_1}\wedge\dots\wedge\mathcal{Z}^{\alpha_p}$, $\mathcal{Z}_{\bar{B}_q} = \mathcal{Z}^{\bar{\beta}_1}\wedge\dots\wedge\mathcal{Z}^{\bar{\beta}_p}$. We have considered here that the coefficients of the forms are functions defined on $E$, i.e., $\psi_{A_p\bar{B}_q} = \psi_{A_p\bar{B}_q}(z,u)$ and $\phi_{A_p\bar{B}_q} = \phi_{A_p\bar{B}_q}(z,u)$, as in the following we will consider the integrals over $E$.

We now define
\begin{equation} \label{inn.prod}
<\Psi,\Phi> = \dfrac{1}{p!q!}\psi_{A_p\bar{B}_q}\overline{\phi^{\bar{A}_pB_q}} = \sum \psi_{A_p\bar{B}_q}\overline{\phi^{\bar{A}_pB_q}},
\end{equation}
where the sum is after $\alpha_1<\dots<\alpha_p$, $\bar{\beta}_1<\dots<\bar{\beta}_q$ and $\phi^{\bar{A}_pB_q} = \phi^{\bar{\alpha}_1\dots\bar{\alpha}_p\beta_1\dots\beta_q} = \phi_{\mu_1\dots\mu_p\bar{\nu}_1\dots\bar{\nu}_q}h^{\bar{\alpha}_1\mu_1}\dots h^{\bar{\alpha}_p\mu_p}h^{\bar{\nu}_1\beta_1}\dots h^{\bar{\nu}_q\beta_q}$. This inner product is independent of the local coordinates, such that $<\Psi,\Phi>$ is a global inner product on $E$. In particular, the "norm" of a form $\Psi$ is defined by
\begin{equation*}
|\Psi|^2 = <\Psi,\Psi> = \dfrac{1}{p!q!}\psi_{A_p\bar{B}_q}\overline{\psi^{\bar{A}_pB_q}}
\end{equation*}

By using the volume form \eqref{f.vol}, we can now define a global inner product on the space of horizontal forms on $\mathcal{T}E$ as
\begin{equation} \label{prod.int.forme}
(\Psi,\Phi) = \int_E <\Psi,\Phi>d\mathcal{V},\quad ||\Psi||^2 = \int_E <\Psi,\Psi>d\mathcal{V}.
\end{equation}

Similarly to the case of complex vector bundles, \cite{Mu,M-P}, we define horizontal differentials of horizontal $(p,q)$-forms by
\begin{align}
\nonumber
(\partial^h\Psi)_{A_{p+1}\overline{B}_q} &= \sum_{i=1}^{p+1} (-1)^{i-1}\delta_{\alpha_i}(\psi_{\alpha_1\dots\hat{\alpha}_i\dots\alpha_{p+1}\overline{B}_q}),\\
\label{dif.h}
(\bar{\partial}^h\Psi)_{A_p\overline{B}_{q+1}} &= (-1)^p\sum_{i=1}^{q+1} (-1)^{i-1}\delta_{\bar{\beta}_i}(\psi_{A_p\bar{\beta}_1\dots\hat{\bar{\beta}}_i\dots\bar{\beta}_{q+1}}).
\end{align}

For the case of K\"ahler Finsler algebroids, we can use the identity \eqref{cond.K} to replace these horizontal derivatives by the horizontal covariant derivatives, that is,
\begin{align*}
(\partial^h\Psi)_{A_{p+1}\overline{B}_q} &= \sum_{i=1}^{p+1} (-1)^{i-1}\nabla_{\mathcal{X}_{\alpha_i}}\psi_{\alpha_1\dots\hat{\alpha}_i\dots\alpha_{p+1}\overline{B}_q},\\
(\bar{\partial}^h\Psi)_{A_p\overline{B}_{q+1}} &= (-1)^p\sum_{i=1}^{q+1} (-1)^{i-1}\nabla_{\mathcal{X}_{\bar{\beta}_i}}\psi_{A_p\bar{\beta}_1\dots\hat{\bar{\beta}}_i\dots\bar{\beta}_{q+1}},
\end{align*}

Following the usual steps in defining a Laplace operator for forms, we now need to introduce the adjoint operators of $\partial^h$ and $\bar{\partial}^h$ with respect to the inner product \eqref{prod.int.forme}. Denote by $\partial^{*h}$ and $\bar{\partial}^{*h}$ the two adjoint operators. We have
\begin{align*}
\partial^{*h}:\mathcal{A}_{p,q}(H\mathcal{T}E)\rightarrow\mathcal{A}_{p,q-1}(H\mathcal{T}E),\quad (\partial^h\Psi,\Phi) = (\Psi,\partial^{*h}\Psi),\\
\bar{\partial}^{*h}:\mathcal{A}_{p,q}(H\mathcal{T}E)\rightarrow\mathcal{A}_{p-1,q}(H\mathcal{T}E),\quad (\bar{\partial}^h\Psi,\Phi) = (\Psi,\bar{\partial}^{*h}\Psi),
\end{align*}
where $\mathcal{A}_{p,q}(H\mathcal{T}E)$ denotes the space of horizontal forms of $(p,q)$-type with compact support on the prolongation algebroid.

We are interested in expressing the adjoint operator $\bar{\partial}^{*h}$. For this purpose, let $\Psi\in\mathcal{A}_{p,q-1}(H\mathcal{T}E)$ and $\Phi\in\mathcal{A}_{p,q}(H\mathcal{T}E)$. Then, a similar computation to the one from \cite{Z-Z} leads to
\begin{align} \label{adj1}
(\bar{\partial}^{*h}\Phi)^{\overline{A}_p\beta_2\dots\beta_q} &= -(-1)^p h^{-2}\delta_{\beta_1}(\phi^{\overline{A}_p\beta_1\dots\beta_q}h^2)\\
\nonumber &= -(-1)^p\sum_{\beta_1} [\delta_{\beta_1}+2\delta_{\beta_1}(\opn{ln}h)]\phi^{\overline{A}_p\beta_1\dots\beta_q},
\end{align}
which, by lowering the indices, gives
\begin{equation} \label{adj2}
(\bar{\partial}^{*h}\Phi)_{A_p\bar{\beta}_2\dots\bar{\beta}_q} = (-1)^{p+1}h^{\bar{\varepsilon}\gamma}\delta_\gamma(\phi_{A_p\bar{\varepsilon}\bar{\beta}_2\dots\bar{\beta}_q}).
\end{equation}

We can now introduce a horizontal Laplace operator, $\Box^h:\mathcal{A}_{p,q}(H\mathcal{T}E)\rightarrow\mathcal{A}_{p,q}(H\mathcal{T}E)$,  by setting
\begin{equation} \label{def.L}
\Box^h = \bar{\partial}^h\circ\bar{\partial}^{*h} + \bar{\partial}^{*h}\circ\bar{\partial}^h.
\end{equation}

The expression of $\Box^h$ is given in the following
\begin{theorem}
The horizontal Laplace operator for a horizontal differential form $\Phi\in\mathcal{A}_{p,q}(H\mathcal{T}E)$ on the prolongation of a Finsler Lie algebroid is given by
\begin{equation} \label{exp.L.forme}
(\Box^h\Phi)_{A_p\overline{B}_q} = -h^{\bar{\varepsilon}\gamma}\bigg(\delta_\gamma\circ\delta_{\bar{\varepsilon}}(\phi_{A_p\overline{B}_q}) - \sum_i(-1)^{i-1}[\delta_\gamma,\delta_{\bar{\beta}_i}]\phi_{A_p\bar{\varepsilon}\bar{\beta}_1\dots\hat{\bar{\beta}}_i\dots\bar{\beta}_q} \bigg).
\end{equation}
\end{theorem}
\begin{proof}
From \eqref{dif.h} and \eqref{adj2} we have
\begin{align*}
(\bar{\partial}^h\circ\bar{\partial}^{*h}\Phi)_{A_p\overline{B}_q} &= -\sum_i(-1)^{i-1}h^{\bar{\varepsilon}\gamma}(\delta_{\bar{\beta}_i}\circ\delta_\gamma)\big(\phi_{A_p\bar{\varepsilon}\bar{\beta}_1\dots\hat{\bar{\beta}}_i\dots\bar{\beta}_q}\big).
\end{align*}
Also,
\begin{align*}
(\bar{\partial}^{*h}\circ\bar{\partial}^h\Phi)_{A_p\overline{B}_q} &= - h^{\bar{\varepsilon}\gamma}\bigg((\delta_\gamma\circ\delta_{\bar{\varepsilon}})(\phi_{A_p\overline{B}_q}) - \sum_i(-1)^i(\delta_\gamma\circ\delta_{\bar{\beta}_i})\big(\phi_{A_p\bar{\varepsilon}\bar{\beta}_1\dots\hat{\bar{\beta}}_i\dots\bar{\beta}_q}\big)\bigg)
\end{align*}
and \eqref{exp.L.forme} follows immediately.
\end{proof}

Let us now consider the case of K\"ahler Finsler algebroids, when using \eqref{int2} yields
\begin{equation*}
\int_E\nabla_{\mathcal{X}_\beta}\bigg(\phi_{A_p\overline{B}_q}\overline{\psi^{\overline{A}_p\beta B_q}}\bigg)d\mathcal{V} = 0
\end{equation*}
and a similar computation as in the case of a K\"ahler Finsler manifold \cite{Z-Z} leads to
\begin{theorem}
On a K\"ahler Finsler algebroid, the horizontal Laplace operator for a horizontal differential form on $\mathcal{T}E$ is
\begin{equation}
\label{lapl.f.K}
(\Box^h\Phi)_{A_p\overline{B}_q} = - h^{\bar{\varepsilon}\gamma}\nabla_{\mathcal{X}_\gamma}\circ\nabla_{\mathcal{X}_{\bar{\varepsilon}}}(\phi_{A_p\overline{B}_q}) + \sum_i h^{\bar{\varepsilon}\gamma}[\nabla_{\mathcal{X}_\gamma},\nabla_{\mathcal{X}_{\bar{\beta}_i}}]\phi_{A_p\bar{\varepsilon}\bar{\beta}_1\dots\hat{\bar{\beta}}_i\dots\bar{\beta}_q}.
\end{equation}
\end{theorem}

\end{document}